\documentclass[a4paper]{amsart}
\usepackage[english]{babel}
\usepackage{amsmath,amssymb,amsthm}
\usepackage{graphicx}

\usepackage[pdftex]{color}
\usepackage[bookmarks=true,hyperindex,pdftex,colorlinks, citecolor=blue]{hyperref}
\hypersetup{pdftitle={Bishop-Phelps-Bollob\'{a}s property for bilinear forms on spaces of continuous functions},
pdfauthor={Sun Kwang Kim, Han Ju Lee, and Miguel Martin},pdfpagemode=UseOutlines}

\theoremstyle{plain}
\newtheorem{theorem}{Theorem}
\newtheorem{corollary}[theorem]{Corollary}

\newtheorem{proposition}[theorem]{Proposition}
\newtheorem{lemma}[theorem]{Lemma}

\theoremstyle{definition}

\newtheorem{definition}[theorem]{Definition}

\DeclareMathOperator{\re}{Re\,}
 
 \DeclareMathOperator{\e}{e}
 \DeclareMathOperator{\Id}{\mathrm{Id}}

\newcommand{\R}{\mathbb{R}}

\newcommand{\inner}[1]{\ensuremath{\left\langle #1\right\rangle}}

\newcommand{\eps}{\varepsilon}

\renewcommand{\leq}{\leqslant}
\renewcommand{\geq}{\geqslant}
\renewcommand{\le}{\leqslant}
\renewcommand{\ge}{\geqslant}

\renewcommand{\tilde}{\widetilde}

\begin{document}

\title{Bishop-Phelps-Bollob\'{a}s property for bilinear forms on spaces of continuous functions}

\dedicatory{Dedicated to the memory of Manuel Valdivia}

\author[Kim]{Sun Kwang Kim}
\address[Kim]{Department of Mathematics, Kyonggi University , Suwon 443-760, Republic of Korea}
\email{\texttt{sunkwang@kgu.ac.kr}}

\author[Lee]{Han Ju Lee}

\address[Lee]{Department of Mathematics Education,
Dongguk University - Seoul, 100-715 Seoul, Republic of Korea}
\email{\texttt{hanjulee@dongguk.edu}}

\author[Mart\'{\i}n]{Miguel Mart\'{\i}n}
\address[Mart\'{\i}n]{Departamento de An\'{a}lisis Matem\'{a}tico,
Facultad de Ciencias,
Universidad de Granada,
E-18071 Granada, Spain}
\email{\texttt{mmartins@ugr.es}}

\subjclass[2010]{Primary 46B20; Secondary 46B04, 46B22}
\keywords{Banach space, approximation, norm attaining operator, bilinear form, Bishop-Phelps-Bollob\'{a}s theorem.}

\thanks{The second named author partially was supported by Basic Science Research Program through the National Research Foundation
of Korea(NRF) funded by the Ministry of Education, Science and Technology (2012R1A1A1006869). Third author partially supported by Spanish MICINN and FEDER project no.~MTM2012-31755 and by Junta de Andaluc\'{\i}a and FEDER grants FQM-185 and P09-FQM-4911.}

\date{September 30th, 2014}

\begin{abstract}
It is shown that the Bishop-Phelps-Bollob\'as theorem holds for bilinear forms on the complex $C_0(L_1)\times C_0(L_2)$ for arbitrary locally compact topological Hausdorff spaces $L_1$ and $L_2$.
\end{abstract}

\maketitle

All along the paper, we will use the following usual notation. Let $B_X$ and $S_X$ denote, respectively, the closed unit ball and the unit sphere of a Banach space $X$, and let $X^*$ denotes the (topological) dual space of $X$. We write $\mathcal{L}(X,Y)$ to denote the space of all bounded linear operators from a Banach space $X$ into a Banach space $Y$ and $\mathcal{B}(X \times Y)$ to denote the space of all bounded bilinear forms defined on $X\times Y$. We say that $T\in \mathcal{L}(X,Y)$ (respectively $B\in \mathcal{B}(X\times Y)$) \emph{attains its norm} if there is $x\in S_X$ such that $\|Tx\|=\|T\|$ (respectively, there are $x\in S_X$ and $y\in S_Y$ such that $|B(x,y)|=\|B\|$).

The classical Bishop-Phelps theorem \cite{BP} states that the set of norm attaining linear functionals are dense in the topological dual of an arbitrary Banach space. Bollob\'as \cite{Bol} gave a quantitative version of this theorem which is now called the Bishop-Phelps-Bollob\'as theorem and states the following (see \cite[Corollary~2.4]{C-K-M-M-R} for this version):

\begin{quote}
Let $X$ be a Banach space. If $x\in B_X$ and $x^*\in S_{B_X^*}$ satisfy $\re x^*(x)> 1- \eps^2/2$. Then there exist $y\in
S_X$ and $y^*\in S_{X^*}$ such that
$y^*(y)=1$, $\|x^*-y^*\|<\eps$ and $\|x-y\|<\eps$.
\end{quote}

In 1963, Lindenstrauss \cite{Lind} studied possible extensions of the Bishop-Phelps theorem to the vector-valued case, starting a fruitful line of research on norm attaining operators. We refer the reader to the expository paper \cite{Acosta-RACSAM} for more information and background. Let us just comment a couple of  interesting results in this line. First, if a Banach space $X$ has the Radon-Nikod\'{y}m property, then norm attaining operators from $X$ into $Y$ are dense in $\mathcal{L}(X,Y)$ for every Banach space $Y$ \cite{Bou}, and this actually characterizes the Radon-Nikod\'{y}m property if it holds for every equivalent renorming of $X$ \cite{Bou}. Second, norm attaining operators are not dense in $\mathcal{L}(L_1[0,1],C[0,1])$ but norm attaining operators are dense in $\mathcal{L}(C(K),L_p(\mu))$ for every compact Hausdorff topological space $K$, every measure $\mu$ and $1\leq p<\infty$ \cite{Sch2}.

In 2008, Acosta, Aron, Garc\'{\i}a and Maestre \cite{AAGM2} started the study of possible vector-valued versions of the Bishop-Phelps-Bollob\'{a}s theorem, introducing the so-called Bishop-Phelps-Bollob\'{a}s property. A pair $(X, Y)$ of Banach spaces has the \emph{Bishop-Phelps-Bollob\'{a}s property for operators} if for every $0<\eps<1$, there is $\eta(\eps)>0$ such that given $T\in \mathcal{L}(X, Y)$ with $\|T\|=1$ and $x_0\in S_X$ satisfying $\|T(x_0)\|>1-\eta(\eps)$, there exist $x_1\in S_X$  and $S\in \mathcal{L}(X, Y)$ such that $\|S(x_1)\|=\|S\|=1$, $\|x_0 - x_1 \|<\eps$ and $\|S- B\|<\eps$. In this case, it also said that the Bishop-Phelps-Bollob\'{a}s theorem holds for $\mathcal{L}(X,Y)$.

It is clear that the Bishop-Phelps-Bollob\'{a}s property for a pair $(X,Y)$ implies that norm attaining operators are dense in $\mathcal{L}(X,Y)$, being false the converse: there is a (reflexive) space $Y$ such that the Bishop-Phelps-Bollob\'{a}s theorem does not holds for $\mathcal{L}(\ell_1^{(2)},Y)$ \cite{AAGM2,ACKLM} ($\ell_1^{(2)}$ is the two-dimensional $L_1$ space) while all elements of $\mathcal{L}(\ell_1^{(2)},Y)$ attain their norms. We refer the reader to \cite{A,AAGM2,ABCCKLLM, ACKLM, ACK, ACGM, CasGuiKad, CK, CKLM, Kim, KL, KLL} and references therein for more information and background. Let us just mention some examples of pairs of classical spaces having the Bishop-Phelps-Bollob\'{a}s property for operators, namely, $(L_1(\mu),L_\infty(\nu))$ for arbitrary measure $\mu$ and localizable measure $\nu$ \cite{ACGM,CKLM}, $(L_1(\mu),L_1(\nu))$ for arbitrary measures $\mu$ and $\nu$ \cite{CKLM}, $(C(K_1),C(K_2))$ for every compact spaces $K_1$ and $K_2$ \cite{ABCCKLLM}, and $(L_p(\mu),Y)$ for arbitrary measure $\mu$, arbitrary Banach space $Y$ and $1<p<\infty$ \cite{ABGM,KL}.

On the other hand, a lot of attention was focus in the 1990's on the problem of extending the Bishop-Phelps theorem to the bilinear case, starting with the paper of Aron, Finet and Werner \cite{AronFinetWerner} where the problem is stated and it is shown, in particular, that the Radon-Nikod\'{y}m property is sufficient to get such a extension. In 1998, Acosta, Aguirre and Pay\'{a} \cite{AAP} found the first negative example. In 1997, Alaminos, Choi, Kim and Pay\'{a} \cite{ACKP} showed that norm attaining bilinear forms on $C_0(L)$ spaces are dense in the space of all bounded bilinear forms. We refer again to the expository paper \cite{Acosta-RACSAM} for a detailed account on the subject.

Let us recall that for all Banach spaces $X$ and $Y$, $\mathcal{B}(X \times Y)$ is isometrically isomorphic to $\mathcal{L}(X,Y^*)$ (by the canonical isometry $B\longmapsto T$ given by $[T(x)](y)=B(x,y)$ for every $y\in Y$ and $x\in X$). Moreover, density of norm attaining bilinear forms on $X\times Y$ implies density of norm attaining operators from $X$ into $Y^*$, as every norm attaining bilinear form produces a norm attaining operator. But the reverse result is far from being true. For instance, norm attaining operators from $L_1[0,1]$ into $L_\infty[0,1]$ are dense in $\mathcal{L}(L_1[0,1], L_\infty[0,1])$ \cite{FinetPaya}, while norm attaining bilinear forms are not dense in $\mathcal{B}(L_1[0,1]\times L_1[0,1])$ \cite{C}.

Very recently, Acosta, Becerra, Garc\'{\i}a and Maestre \cite{ABGM} introduced the Bishop-Phelps-Bollob\'{a}s property for bilinear forms, a concept which appeared without name in \cite{CS}.

\begin{definition}[\mbox{\cite{ABGM,CS}}]\label{definition}
A pair $(X, Y)$ of Banach spaces have the \emph{Bishop-Phelps-Bollob\'as property for bilinear forms} if for every $0<\eps<1$, there is $\eta(\eps)>0$ such that given $B\in \mathcal{B}(X\times Y)$ with $\|B\|=1$ and $(x_0,y_0) \in S_X\times S_Y$ satisfying $\|B(x_0,y_0)\|>1-\eta(\eps)$, there exist $(x_1,y_1)\in S_X\times S_Y$  and $S\in \mathcal{B}(X\times Y)$ with $\|S\|=1$ satisfying the following conditions:
\[
|S(x_1,y_1)|=1, \qquad \|x_0 - x_1 \|<\eps, \quad \|y_0 -y_1 \|<\eps, \quad \text{and} \quad \|S- B\|<\eps.
\]
In this case, we also say that the Bishop-Phelps-Bollob\'{a}s theorem holds for $\mathcal{B}(X\times Y)$.
\end{definition}

There are some recent results about the Bishop-Phelps-Bollob\'{a}s property for bilinear forms \cite{ABGM, ABCGKLM, CS, GLM, KL}. For instance, let us comment that the Bishop-Phelps-Bollob\'{a}s theorem holds for $\mathcal{B}(X\times Y)$ if the Banach spaces $X$ and $Y$ are uniformly convex \cite{ABGM}.

If a pair $(X,Y)$ of Banach spaces has the Bishop-Phelps-Bollob\'{a}s for bilinear forms, then norm attaining bilinear forms on $X\times Y$ are dense and, on the other hand, the pair $(X,Y^*)$ has the Bishop-Phelps-Bollob\'{a}s property for operators. None of the reversed results are true: the pair $(\ell_1,\ell_1)$ fails the Bishop-Phelps-Bollob\'{a}s for bilinear forms \cite{CS}, while norm attaining bilinear forms are dense in $\mathcal{B}(\ell_1\times \ell_1)$ \cite{AronFinetWerner} and also the pair $(\ell_1,\ell_\infty)$ has the Bishop-Phelps-Bollob\'{a}s property for operators \cite{AAGM2}.

Recently, it has been shown that the Bishop-Phelps-Bollob\'{a}s theorem holds for $\mathcal{L}(C_0(L), L_1(\mu))$ in the complex case for arbitrary locally compact Hausdorff topological space $L$ and arbitrary positive measure $\mu$ \cite{A, KLL}. As $C_0(L)^*$ is isometrically isomorphic to a $L_1(\mu)$ space, it is natural to ask whether the Bishop-Phelps-Bollob\'{a}s theorem holds for  $\mathcal{B}(C_0(L_1)\times C_0(L_2))$. In this paper, we show that the answer to this question is affirmative (again in the complex case). Concretely, we will prove the following result.

\begin{theorem}\label{main:theorem}
Let $L_1$ and $L_2$ locally compact Hausdorff topological spaces. Then the pair $(C_0(L_1), C_0(L_2))$ has the Bishop-Phelps-Bollob\'{a}s property for bilinear forms in the complex case. Moreover, for every $0<\eps<1$ there exists $\eta(\eps)>0$, independent of $L_1$ and $L_2$, which is valid for $\mathcal{B}(C_0(L_1)\times C_0(L_2))$ in Definition~\ref{definition}.
\end{theorem}

Even in the particular case of $c_0$, the above result is new. It  solves in the positive \cite[Open problem 4.6.(1)]{ABCGKLM} and \cite[Problem (b) in p.\ 385]{KL} for the complex case.

\begin{corollary}
The Bishop-Phelps-Bollob\'{a}s theorem holds for $\mathcal{B}(c_0 \times c_0)$ in the complex case.
\end{corollary}

The rest of this paper is devoted to provide the proof of Theorem~\ref{main:theorem}, for which we will need a number of preliminary results. From now on, we will only consider complex Banach spaces.

First, we fully use the power of the uniform complex convexity of $L_1(\mu)$ spaces. The \emph{modulus of complex convexity} $H_X$ of a complex Banach space $X$ is defined by
\[
H_X(\eps) = \inf\left\{ \sup_{0\le \theta \le 2\pi} \left\|x + \e^{i\theta} y \right\| -1\, :\ x\in S_X,\, \|y\|\ge \eps\right\} \qquad \bigl(\eps\in \R^+\bigr).
\]
A (complex) Banach space is said to be \emph{uniformly complex convex} if $H_X(\eps)>0$ for all $\eps>0$. This concept has been recently used to study denseness of norm attaining operators \cite{A,CKL}. It is easy to check that every uniformly convex (complex) Banach space is uniformly complex convex. The converse result is false, as it is shown in \cite{Glo} that $L_1(\mu)$ is uniformly complex convex. Moreover, there is common lower bound for the modulus of complex convexity of all $L_1(\mu)$ spaces \cite[Theorem 1]{Glo} (we also refer to \cite{Lee1, Lee2}).

\begin{lemma}[\mbox{\cite{Glo}}]\label{function-eta_0}
There is a function $\eta_0:(0,1) \longrightarrow (0, \infty)$ such that $H_{L_1(\mu)}(\eps) \geq \eta_0(\eps)$ for every $\eps\in (0, 1) $ and for every measure $\mu$.
\end{lemma}

Second, let $L$ be a locally compact Hausdorff space and $C_0(L)$ be the Banach space consisting of all complex-valued continuous functions which vanish at infinity. Recall that a bounded linear functional $x^*$ on $C_0(L)$ is represented by a complex-valued Radon measure $\mu_{x^*}$ on $L$. That is, for each $x\in C_0(L)$,
\[
x^*(x) = \int_L x \, d\mu_{x^*}
\]
(this is the classical Riesz representation theorem for $C_0(L)^*$). Notice that every complex-valued Radon measure is regular and also that there is a Borel measure $\nu$ on $L$ such that  $C_0(L)^*$ is isometrically isometric to $L_1(\nu)$ (just take a maximal family $\{\mu_\alpha\}_\alpha$ of mutually singular positive Radon measures, then $C_0(K)^* =\left[ \oplus_{\alpha} L_1(\mu_\alpha)\right]_{\ell_1}$) and so $C_0(L)^{**}$ can be identified with $L_\infty(\nu)$. We refer to the classical book \cite{Rud} for background. For each Borel subset $A$, we define the operator $P_A: C_0(L)^{**}\longrightarrow C_0(L)^{**}$ by $P_A(f) = f\chi_A$ for every $f\in C_0(L)^{**}\equiv L_\infty(\nu)$. Given a Radon measure $\mu$ on $L$ and a Borel subset $A$ of $L$, $|\mu|(A)$ denotes the total variation of $\mu$ on $A$. We will use the following result which is just a particular case of \cite[Lemma~2.3]{A} using that $C_0(L_2)^*\equiv L_1(\nu)$ is uniformly complex convex and taking $\eta_0$ as the function given by Lemma~\ref{function-eta_0}.

\begin{lemma}[\mbox{Particular case of \cite[Lemma 2.3]{A}}]\label{lemma1}
Let $L_1$, $L_2$ be locally compact Hausdorff spaces and let $A$ be a Borel subset of $L_1$. For given $0<\lambda<1$, if $T\in \mathcal{L}(C_0(L_1), C_0(L_2)^*)$ with $\|T\|=1$ satisfies $\|T^{**}P_A\|>1-{{\eta_0(\lambda)}\over{1+\eta_0(\lambda)}}$, then $\|T^{**}(\Id-P_A)\|\leq \lambda$.
\end{lemma}

The next result is clearly well-known (see \cite[5.5.4 and G.5]{Albiac-Kalton} for instance) and follows from the fact that $C_0(L)^*\equiv L_1(\nu)$ does not contains $c_0$. We state it for the sake of completeness.

\begin{lemma}\label{lemma-bidual}
Let $L_1$ and $L_2$ be locally compact Hausdorff spaces and $T\in \mathcal{L}(C_0(L_1),C_0(L_2)^*)$. Then $T$ is weakly compact and so $T^{**}(C_0(L_1)^{**})\subset C_0(L_2)^*$.
\end{lemma}

We know present the main ingredient in the proof of Theorem~\ref{main:theorem}.

\begin{lemma} \label{lemlem} Let $L_1$ and $L_2$ be locally compact Hausdorff spaces. Then for every $0<\lambda<1$ there is $\eta>0$ which satisfies the following:

\noindent If $B\in \mathcal{B}(C_0(L_1)\times C_0(L_2))$ with $\|B\|=1$ and $(x,y) \in S_{C_0(L_1)}\times S_{ C_0(L_2)}$ satisfy $|B(x,y)|>1-\eta^4$, then there exist compact sets $K_1\subset L_1$, $K_2\subset L_2$  and a bilinear form $\tilde{B}\in \mathcal{B}(C(K_1)\times C(K_2))$ such that
$$
\left|\tilde{B}(x_{|K_1}, y_{|K_2})\right|>1-\lambda^2 \quad \text{and} \quad \|C-B\|<\lambda,
$$
where $C\in \mathcal{B}(C_0(L_1)\times C_0(L_2))$ is given by $C(u,v)=\tilde{B}(u_{|K_1},v_{|K_2})$ for every $(u,v)\in C_0(L_1) \times C_0(L_2)$. Besides, $\min\limits_{t\in K_1} \left|x(t)\right|\ge 1-\eta^2$ and $\min\limits_{s\in K_2} \left|y(s)\right|\ge 1-\eta$.
\end{lemma}

\begin{proof}
Choose $0<\eta<1$ such that $\eta^4+\eta^3+\eta^2+\eta<\min\left\{\lambda^2,{{\eta_0\left({{\lambda}\over{2}}\right)}\over{1+ \eta_0\left({{\lambda}\over{2}}\right)}} \right\}$, where $\eta_0(\cdot)$ is given by Lemma~\ref{function-eta_0}.

Fix $B\in \mathcal{B}(C_0(L_1)\times C_0(L_2))$ with $\|B\|=1$ and  $(x,y) \in S_{C_0(L_1)}\times S_{ C_0(L_2)}$ satisfying $|B(x,y)|>1-\eta^4$. Consider $T\in \mathcal{L}(C_0(L_1),C_0(L_2)^*)$ the operator defined by $[T(u)](v)=B(u,v)$ for every $v\in C_0(L_2)$ and $u\in C_0(L_1)$, and write
\[
A_1=\big\{t\in L_1 ~:~  |x(t)| \geq 1-\eta^2\},\quad A_2=\big\{s\in L_2 ~:~ |y(s)| \geq 1-\eta \}.
\]
Then, we have that
\begin{align*}
1-\eta^4
<&\left|[T(x)](y)\right|=\left|[T^*(y)](x)\right|=\left|\int_L x \, d \mu_{T^*(y)}\right|\\
\leq& \int_{A_1} d\left|\mu_{T^*(y)}\right|+(1-\eta^2)\int_{A_1^c} d\left|\mu_{T^*(y)}\right|\leq 1-\eta^2 \int_{A_1^c} d\left|\mu_{T^*(y)}\right|.\\
\end{align*}
Hence, $\left|\mu_{T^*(y)}\right|(A_1^c)<\eta^2$. By the regularity of $\mu_{T^*(y)}$, there exists a compact set $K_1\subset A_1$ such that $\left|\mu_{T^*(y)}\right|(K_1^c)<\eta^2$.
Therefore,
\begin{align*}
\left|\int_{K_1} x\, d\mu_{T^*(y)}\right|=&\left|T^{**}(x_{|K_1})(y)\right|\\
\geq&\left|[T^{**}(x)](y)\right|-\left|[T^{**}(x)](y)- [T^{**}(x_{|K_1})](y)\right|\\
\geq& 1-\eta^4-\left|\int_{K_1^c}x\, d\mu_{T^*(y)}\right|\\
>&1-\eta^4-\eta^2
\end{align*}
Thus, $\|T^{**}P_{K_1}\|>1-\eta^4-\eta^2>1-{{\eta_0\left({{\lambda}\over{2}}\right)}\over{1+ \eta_0\left({{\lambda}\over{2}}\right)}}$.
By Lemma~\ref{lemma1}, we get  $\|T^{**}-T^{**}P_{K_1}\|\leq \lambda/2$. Now, let $U\in \mathcal{L}(C_0(L_2),C_0(L_1)^*)$ be the operator defined by $[U(v)](u)=T^{**}(P_{K_1}u)(v)$ for every $u\in C_0(L_1)$ and $v\in C_0(L_2)$ (note that $T^{**}P_{K_1}(C_0(L_1))\subset C_0(L_2)^{*}$ by Lemma~\ref{lemma-bidual}). Since
$$
|[U(y)](x)|=\left|\int_{K_1} x \, d\mu_{T^*(y)}\right|>1-\eta^4 -\eta^2,
$$
similarly to the above, we get
$\left|\mu_{U^*(x)}\right|(A_2^c)<\eta^3+\eta$. So there exists a compact set  $K_2\subset A_2$ such that
 $\left|\mu_{U^*(x)}\right|(K_2^c)<\eta^3+\eta.$
Hence, $$\left|\int_{K_2^c} y\, d\mu_{U^*(x)}\right|>1-\eta^4-\eta^2-\eta^3-\eta$$ and $$\|U^{**}P_{K_2}\|>1-{{\eta_0\left({{\lambda}\over{2}}\right)}\over{1+ \eta_0\left({{\lambda}\over{2}}\right)}}.$$
Using Lemma~\ref{lemma1} again, we get  $\|U^{**}-U^{**}P_{K_2}\|\leq \lambda/2$.

Define $\tilde{B}\in \mathcal{B}(C(K_1)\times C(K_2))$ by $\tilde{B}(u',v')=[U^{**}(P_{K_2}\tilde{v}')](\tilde{u}')$ for every $(u',v')\in C(K_1)\times C(K_2)$, where $\tilde u'$ and $\tilde v'$ are norm preserving extensions of $u'$ and $v'$ to $L_1$ and $L_2$, respectively. Then, $C$ as in the statement of the lemma is well-defined. Indeed, if $(u,v)\in C_0(L_1)\times C_0(L_2)$, then $\inner{v, U^*u} = \inner{T^{**}(P_{K_1}u), v}.$
Since $U^*u\in C_0(L_2)^*$ and $T^{**}(P_{K_1}u)\in C_0(L_2)^*$, we have, by Goldstine's lemma,
$\inner{\tilde{v}, U^*u} = \inner{T^{**}(P_{K_1}u), \tilde{v}}$
for all $\tilde{v}\in C_0(L_2)^{**}$. Therefore,
\[
 U^{**}(P_{K_2}\tilde{v}')(\tilde{u}')  = \inner{ P_{K_2}\tilde{v'}, U^*(\tilde{u'}) }  = \inner{ P_{K_2}\tilde{v'}, T^{**}(P_{K_1}\tilde{u'}) }.
\]
This implies that $C$ is well-defined.
Finally, we also have that
\begin{align*}
 \|B-C\|
 =&\sup_{x\in S_{C_0(L_1)},~y\in S_{C_0(L_2)}}|B(x,y)-C(x,y)|\\
 =&\sup_{x\in S_{C_0(L_1)},~y\in S_{C_0(L_2)}}|T^{**}(x)(y)-U^{**}P_{K_2}(y)(x)|\\
 =&\sup_{x\in S_{C_0(L_1)},~y\in S_{C_0(L_2)}}\left(|T^{**}(x)(y)-T^{**}P_{K_1}(x)(y)|+ |T^{**}P_{K_1}(x)(y)-U^{**}P_{K_2}(y)(x)|\right)\\
 \leq&\sup_{x\in S_{C_0(L_1)},~y\in S_{C_0(L_2)}}\left(|T^{**}(x)(y)-T^{**}P_{K_1}(x)(y)|+ |U^{**}(y)(x)-U^{**}P_{K_2}(y)(x)|\right)\\
=& \|T^{**}-T^{**}P_{K_1}\|+\|U^{**}-U^{**}P_{K_2}\|\leq \lambda.\qedhere
\end{align*}
\end{proof}

The next ingredient is the following easy consequence of Urysohn lemma.

\begin{lemma}\label{function}
Assume that two continuous functions $f,g\in B_{C_0(L)}$ and a compact set $A\subset L$ are given. If $\|f_{|_A}-g_{|_A}\|<\eps$, then there exists a function $h\in B_{C_0(L)}$ such that $h_{|_A}=f_{|_A}$ and $\|h-g\|<\eps$.
\end{lemma}

\begin{proof}
Since $A$ is compact there exists an open set $U$ so that $A\subset U$ and $\|f_{|_U}-g_{|_U}\|<\eps$.
Using Urysohn lemma, choose $u\in C_0(L)$ such that $0\leq u \leq 1$, $u=1$ on $A$ and $u=0$ on $U^c$.
Then, clearly $h=uf+(1-u)g$ works.
\end{proof}

It is proved in \cite{ACKP} that the set of all norm attaining bilinear forms are dense in $\mathcal{B}(C_0(L)\times C_0(L))$ for every locally compact space $L$. With a slight modification of the proof we may get the following.

\begin{proposition}[\mbox{\cite{ACKP}}]\label{prop:density} Let $L_1$, $L_2$ be locally compact Hausdorff topological spaces. Then, every continuous bilinear form on $C_0(L_1)\times C_0(L_2)$ can be approximated by norm attaining bilinear forms.
\end{proposition}

We are now ready to provide the proof of the main result.

\begin{proof}[Proof of Theorem~\ref{main:theorem}]
Given $0<\eps<1$,  let
$$
\psi= \min \left\{{{\eps^2}\over{4}}, {{\eta_0\left({{\eps}\over{4}}\right)}\over{1+ \eta_0\left({{\eps}\over{4}}\right)}} \right\}.
$$
Choose suitable positive numbers $0<\lambda<\zeta<\gamma<\psi$ satisfying the following three conditions:
 \begin{enumerate}
 \item ${{\left((1-\lambda^2)^2 -2\lambda\right)\left(1+\lambda-\lambda^3\right)}\over{1+\lambda}}-{{\lambda^2}\over{2}}>1-\zeta^2$.
 \item $\zeta < \min \left\{\gamma, {{\eta_0 \left(\gamma\right)}\over{1+ \eta_0\left(\gamma\right)}} \right\}$.
 \item ${{\left((1-\lambda^2)^2 -2\lambda\right)\left(1+\lambda-\lambda^3\right)}\over{1+\lambda}}-{{\lambda^2}\over{2}}-\zeta - 2\gamma>1-\psi^2$.
 \end{enumerate}
For such $\lambda$, choose $0<\eta<\eps$ given by Lemma~\ref{lemlem}.

Fix $B\in \mathcal{B}\left(C_0(L_1)\times C_0(L_2)\right)$ with $\|B\|=1$ and $(x_0,y_0)\in S_{C_0(L_1)}\times S_{C_0(L_2)}$ satisfying
$$
\left|B(x_0,y_0)\right|>1-\eta^4.
$$
Then there are compact sets $K_1$, $K_2$ and a bilinear form $\tilde{B}\in \mathcal{B}(C(K_1)\times C(K_2))$ which satisfy the conditions in Lemma~\ref{lemlem}. Let $C\in \mathcal{B}(C_0(L_1)\times C_0(L_2))$ be the canonical extension of $\tilde{B}$ and let $f_0={x_0}_{|K_1}$ and $g_0={y_0}_{|K_2}$. Choose $\alpha\in \mathbb{C}$ with $|\alpha|=1$ such that $|\tilde{B}(f_0, g_0)|=\alpha \tilde{B}(f_0, g_0)$.

Define $\tilde{B}_1 \in \mathcal{B}(C(K_1)\times C(K_2))$ by
$$
\tilde{B}_1(f, g)=\alpha \tilde{B}(f,g)+\lambda \alpha^2 \tilde{B}(f,g_0)\tilde{B}(f_0,g) \qquad \bigl((f,g)\in C(K_1)\times C(K_2) \bigr)
$$
and write $\tilde{B}_2=\tilde{B}_1/\|\tilde{B}_1\|$. Since the set of norm attaining bilinear mappings is dense in $\mathcal{B}(C(K_1)\times C(K_2))$ (Proposition~\ref{prop:density}), there exist $\tilde{B}_3\in \mathcal{B}(C(K_1)\times C(K_2))$ and $(f_1,g_1)\in S_{C(K_1)}\times S_{C(K_2)}$ such that
$$
\|\tilde{B}_3\|=1,~\|\tilde{B}_3- \tilde{B}_2\|<{{\lambda^2}\over{2}} \quad \text{and} \quad |\tilde{B}_3(f_1,g_1)|=1.
$$
Moreover, rotating $g_1$ and $f_1$ if needed, we may assume that
$$
\tilde{B}_3(f_1, g_1)=1 \quad \text{and} \quad \alpha \tilde{B}(f_1, g_0)=|\tilde{B}(f_1, g_0)|.
$$
First, we have that
\begin{align*}
\|\tilde{B}_1\|
&\geq |\tilde{B}_1(f_0, g_0)|=\left|\alpha \tilde{B}(f_0,g_0)+\lambda \left(\alpha \tilde{B}(f_0,g_0)\right)^2\right|\\
&=\left|\tilde{B}(f_0,g_0)\right|+\lambda \left|\tilde{B}(f_0,g_0)\right|^2\geq 1-\lambda^2+\lambda (1-\lambda^2)^2.
\end{align*}
Thus, $1-\lambda\leq \|\tilde{B}_1\|\leq 1+\lambda$.
Second,
\begin{align*}
1=\|\tilde{B}_3\|
&= \tilde{B}_3(f_1, g_1)\leq \re \tilde{B}_2(f_1,g_1)+{{\lambda^2}\over{2}}={{\re \tilde{B}_1(f_1,g_1)}\over{\|\tilde{B}_1\|}}+{{\lambda^2}\over{2}}\\
&={{1}\over{\|\tilde{B}_1\|}}\left(\re \tilde{B}(f_1, g_1)+\re \lambda  \alpha^2 \tilde{B}(f_1, g_0)\cdot \tilde{B}(f_0, g_1)\right)+{{\lambda^2}\over{2}}\\
&\leq {{1}\over{\|\tilde{B}_1\|}}\left(1+\re \lambda \alpha \tilde{B}(f_0, g_1)\right)+{{\lambda^2}\over{2}}.
\end{align*}
Hence, we have $ 1+\re \lambda \alpha \tilde{B}(f_0, g_1)\geq \|\tilde{B}_1\|-{{\lambda^2}\over{2}}\|\tilde{B}_1\|\geq \|\tilde{B}_1\|-\lambda^2$, and this implies
\[
\re \alpha \tilde{B}(f_0, g_1) \geq (1-\lambda^2)^2 -2\lambda.
\]
In the same way, we get
\[
\alpha\tilde{B}(f_1, g_0)=|\tilde{B}(f_1, g_0)|\geq (1-\lambda^2)^2 -2\lambda.
\]
Define $T_1\in \mathcal{L}(C(K_1),C(K_2)^*)$ by $[T_1(f)](g)=\tilde{B}_3(f,g)$ for $g\in C(K_2)$ and $f\in C(K_1)$. We have
\begin{align*}
\re T_1(f_0)(g_1)
&=\re \tilde{B}_3(f_0,g_1)\geq \re \tilde{B}_2(f_0,g_1)-{{\lambda^2}\over{2}}\\
&= {{\re \tilde{B}_1(f_0,g_1)}\over{\|\tilde{B_1}\|}}-{{\lambda^2}\over{2}}\geq  {{\re \tilde{B}_1(f_0,g_1)}\over{1+\lambda}}-{{\lambda^2}\over{2}}\\
&= {{\re \alpha \tilde{B}(f_0,g_1)+\re \lambda \alpha^2 \tilde{B}(f_0,g_0)\tilde{B}(f_0,g_1)}\over{1+\lambda}}-{{\lambda^2}\over{2}}\\
&={{\re \alpha \tilde{B}(f_0,g_1)\left(1+\lambda \alpha \tilde{B}(f_0,g_0)\right)}\over{1+\lambda}}-{{\lambda^2}\over{2}}\\
&>{{\left((1-\lambda^2)^2 -2\lambda\right)\left(1+\lambda-\lambda^3\right)}\over{1+\lambda}}-{{\lambda^2}\over{2}}\\
&>1-\zeta^2.
\end{align*}
By the polar decomposition of the measure $\mu_{T^*_1(g_1)}$, there is a measurable function $h_{g_1}\in C(K_1)^{**}$ with $|h_{g_1}|=1$ such that
\[
[T^*_1(g_1)](f)=\int_{K_1}f h_{g_1} d|\mu_{T^*_1(g_1)}|\qquad \bigl(f\in C(K_1)\bigr).
\]
Consider the set $A_1=\left\{t \in K_1\, :\, \re \left({{f_0(t)+f_1(t)}\over{2}}\right)h_{g_1}(t)\geq 1-\zeta \right\}$ and observe that
\begin{align*}
1-\zeta^2
&< \re \left[T_1 \left({{f_0+f_1}\over{2}}\right)\right](g_1) = \re [T_1^*(g_1)]\left({{f_0+f_1}\over{2}}\right) =\int_{K_1}\re \left({{f_0+f_1}\over{2}}\right)h_{g_1} d|\mu_{T^*_1(g_1)}|\\
&=\int_{A_1}\re \left({{f_0+f_1}\over{2}}\right)h_{g_1} d|\mu_{T^*_1(g_1)}|+\int_{K_1\setminus A_1}\re \left({{f_0+f_1}\over{2}}\right)h_{g_1} d|\mu_{T^*_1(g_1)}|\\
&<\int_{A_1} d|\mu_{T^*_1(g_1)}|+(1-\zeta)\int_{K_1\setminus A_1} d|\mu_{T^*_1(g_1)}|=1-\zeta \int_{K_1\setminus A_1} d|\mu_{T^*_1(g_1)}|
\end{align*}
Hence, $|\mu_{T^*_1(g_1)}|(K_1\setminus A_1)<\zeta$. By regularity, there is a compact subset $F_1$ of $A_1$ such that $|\mu_{T^*_1(g_1)}|(K_1\setminus F_1)<\zeta$, which implies
\begin{align*}
\|T_1^{**}P_{F_1}\|&\ge  \sup_{f\in B_{C(K_1)}} [T^{**}_1(P_{F_1}f)](g_1)  = \sup_{f\in B_{C(K_1)}} \int_{F_1} f\, d\mu_{T^*_1(g_1)}= |\mu_{T^*_1(g_1)}|(F_1)\\
& > |\mu_{T^*_1(g_1)}|(K_1) - \zeta = \|T^*(g_1)\|-\zeta=1-\zeta\geq 1-{{\eta_0\left(\gamma\right)}\over{1+ \eta_0\left(\gamma\right)}}.
\end{align*}
Hence, from Lemma~\ref{lemlem}, we have $\|T_1^{**}-T_1^{**}P_{F_1}\|\leq \gamma$. Using the definition of the set $A_1$, it follows that
\begin{equation}
\label{func}
\|f_{0|_{A_1}}-f_{1|_{A_1}}\|<2 \sqrt{\zeta(1-\zeta)}.
\end{equation}
Indeed, for $t\in A_1$, we know that $|\re f_0(t) + \re f_1(t)| \geq 2 - 2 \zeta$. As $|f_0(t)|\leq 1$ and $|f_1(t)|\leq 1$, it follows by the parallelogram law, that
$$
|f_0(t)-f_1(t)|^2 \leq 4 - |f_0(t)+f_1(t)|^2 \leq 4 - |\re f_0(t) + \re f_1(t)| \leq 4 (\zeta - \zeta^2).
$$
Now, using Urysohn lemma we can find $f_2\in S_{C(K_1)}$ satisfying
$$
\|f_1-f_2\|<3\zeta,\quad [T_1(f_2)](g_1)=1, \quad \text{and} \quad |f_2(t)|=1 \text{ for every $t\in F_1$}.
$$
Indeed, consider $U=\left\{t\in K_1\, :\, \left| f_1(t)\right|>1-3\zeta\right\}$ and observe that $F_1\subset A_1\subset U$, so there exists a function $u\in C(K_1)$ such that $0\leq u \leq 1$, $u(t)=1$ for every $t\in F_1$, and $u(t)=0$ for every $t\in K_1\setminus U$. Define the function $f_2$ by $f_2(t)=u(t){{f_1}(t)\over{|f_1(t)|}}+(1-u(t))f_1(t)$ for every $t\in U$ and $f_2(t)=f_1(t)$ for every $t\in K_1\setminus U$. Similarly, define $f_3$ by $f_3(t)=-u(t){{f_1}(t)\over{|f_1(t)|}}+(1+u(t))f_1(t)$ for every $t\in U$ and $f_2(t)=f_1(t)$ for every $t\in K_1\setminus U$. It is obvious that $f_2,f_3\in B_{C(K_1)}$ and $\|f_2-f_1\|<3\zeta$. Moreover, $[T_1(f_2)](g_1)=1$ since $[T_1(f_2+f_3)](g_1)=2[T_1(f_1)](g_1)=2$ and $|[T_1(f_2)](g_1)|, |[T_1(f_3)](g_1)|\leq 1$.

Choose any $t_1\in F_1$ and define $S_1\in \mathcal{L} (C(K_1), C(K_2)^*)$ by
$$
S_1(f)=T_1^{**}P_{F_1}(f)+ \overline{f_2(t_1)}f(t_1)\bigl(T_1^{**}-T_1^{**}P_{F_1}\bigr)(f_2) \qquad \bigl( f\in C(K_1) \bigr).
$$
This is well-defined by Lemma~\ref{lemma-bidual}, and satisfies
$$
\|S_1\|=[S_1(f_2)](g_1)=1 \quad \text{and} \quad \|S_1-T_1\|<2\gamma.
$$
Let $T_2\in \mathcal{L}(C(K_2),C(K_1)^*$ be given by $[T_2(g)](f)=[S_1(f)](g)$ for $f\in C(K_1)$ and $g\in  C(K_2)$. We have
\begin{align*}
\re [T_2(g_0)](f_2)
&=\re [S_1(f_2)](g_0)\geq \re [T_1(f_2)](g_0) - 2\gamma = \re [T_1(f_1)](g_0) -\zeta - 2\gamma\\
&=\re \tilde{B}_3(f_1,g_0)\geq \re \tilde{B}_2(f_1,g_0)-{{\lambda^2}\over{2}}-\zeta - 2\gamma\\
&\geq {{\re \tilde{B}_1(f_1,g_0)}\over{\|\tilde{B_1}\|}}-{{\lambda^2}\over{2}}\geq  {{\re \tilde{B}_1(f_1,g_0)}\over{1+\lambda}}-{{\lambda^2}\over{2}}-\zeta - 2\gamma\\
&= {{\re \alpha \tilde{B}(f_1,g_0)+\re \lambda \alpha^2 \tilde{B}(f_1,g_0)\tilde{B}(f_0,g_0)}\over{1+\lambda}}-{{\lambda^2}\over{2}}-\zeta - 2\gamma\\
&={{\re \alpha \tilde{B}(f_1,g_0)\left(1+\lambda \alpha \tilde{B}(f_0,g_0)\right)}\over{1+\lambda}}-{{\lambda^2}\over{2}}-\zeta - 2\gamma\\
&>{{\left((1-\lambda^2)^2 -2\lambda\right)\left(1+\lambda-\lambda^3\right)}\over{1+\lambda}}-{{\lambda^2}\over{2}}-\zeta - 2\gamma >1-\psi^2
\end{align*}
With the same procedure than above, we get a measurable function $h_{f_2}\in C(K_2)^{**}$ with $|h_{f_2}|=1$ such that
$$
[T^*_2(f_2)](g)=\int_{K_2} g h_{f_2} d|\mu_{T^*_2(f_2)}| \qquad \bigl(g\in C(K_2) \bigr),
$$
considering the set $A_2=\left\{t \in K_2 \,:\, \re \left({{g_0+g_1}\over{2}}\right)h_{f_2}\geq 1-\psi \right\}$, we get that $\bigl\|g_{0|_{A_2}}-g_{1|_{A_2}}\bigr\| <2\sqrt{\psi(1-\psi)}$ as in the estimation (\ref{func}) and there is a compact subset $F_2$ of $A_2$ satisfying
\[
\|T_2^{**}-T_2^{**}P_{F_2}\|\leq {{\eps}\over{4}}.
\]
Besides, there exists a function $g_2\in S_{C(K_2)}$ satisfying
$$
\|g_1-g_2\|<3\psi,\quad [T_2(g_2)](f_2)=1, \quad \text{and} \quad |g_2(t)|=1 \text{ for every $t\in F_2$.}
$$
Choose any $t_2\in F_2$, and define an operator $S_2\in \mathcal{L}(C(K_2),C(K_1)^*)$ by
$$
S_2(g)=T_2^{**}P_{F_2}(g)+ \overline{g_2(t_2)}g(t_2)\bigl(T_2^{**}-T_2^{**}P_{F_2}\bigr)(g_2) \qquad \bigl(g\in C(K_2)\bigr).
$$
Then, $\|S_2\|=[S_2(g_2)](f_2)=1$, and $\|S_2-T_2\|<{{\eps}\over{2}}$.

Finally, let $\tilde{B}_4 \in \mathcal{B}(C(K_1)\times C(K_2))$ be given by $\tilde{B}_4(f,g)=[S_2(g)](f)$ for every $(f,g)\in C(K_1)\times C(K_2)$ and let $D\in \mathcal{B}(C_0(L_1)\times C_0(L_2))$ be the canonical extension of $\tilde{B}_4$. Notice that $|D(u,v)|=1$ for any extensions $(u,v)\in S_{C_0(L_1)}\times S_{C_0(L_2)}$ of $(f_{2|_{A_1}},g_{2|_{A_2}})$ because of our construction.
Let $x_1\in C_0(L_1)$ and $y_1\in C_0(L_2)$ be any norm preserving extensions of $f_{2|_{A_1}}$ and $g_{2|_{A_2}}$ respectively. As
$$
\|x_{0|_{A_1}}-x_{1|_{A_1}}\|<2\sqrt{\zeta(1-\zeta)}<\eps \quad \text{and} \quad \|y_{0|_{A_2}}-y_{1|_{A_2}}\|<2\sqrt{\psi(1-\psi)}<\eps,
$$
Lemma~\ref{function} provides with $x_2\in S_{C_0(L_1)}$ and $y_2\in S_{C_0(L_2)}$ such that $x_{2|_{A_1}}=x_{0|_{A_1}}$, $y_{2|_{A_2}}=y_{0|_{A_2}}$ and
$$
\|x_2-x_0\|,\,\|y_2-y_0\|<\eps.
$$
Observe that $|D(x_2,y_2)|=1$. Finally,
\begin{align*}
\left\|B-\overline{\alpha}D\right\|& =\left\|D-\alpha B\right\|
< \|D-\alpha C\|+\|\alpha C-\alpha B\|\leq \|\tilde{B}_4-\alpha \tilde{B} \|+\lambda\\
&\leq \|\tilde{B}_4-\tilde{B}_3\|+\|\tilde{B}_3-\tilde{B}_2 \|+\|\tilde{B}_2-\tilde{B}_1\|+\|\tilde{B}_1-\alpha \tilde{B}\|+\lambda\\
&\leq \|S_2-T_2\|+\|S_1-T_1\|+{{\lambda^2}\over{2}}+\lambda+\lambda+\lambda\\
&<2\gamma+{{\eps}\over{2}}+{{\lambda^2}\over{2}}+3\lambda<\eps.\qedhere
\end{align*}
\end{proof}

\vspace*{1cm}

\end{document}